\documentclass{article}

\usepackage{mathptmx}
\usepackage{eucal}
\usepackage{amsmath}
\usepackage{amscd}
\usepackage{amssymb}
\usepackage{amsthm}
\usepackage{xspace}
\usepackage[all,tips]{xy}
\usepackage[dvips]{graphicx}
\usepackage{verbatim}
\usepackage{syntonly}
\usepackage{hyperref}

\theoremstyle{plain}
\newtheorem{thm}{Theorem}[section]

\newtheorem{prop}[thm]{Proposition}
\newtheorem{lemma}[thm]{Lemma}

\theoremstyle{definition}

\newenvironment{pf}
{\begin{proof}} {\end{proof}}

\DeclareMathOperator{\Aut}{Aut} \DeclareMathOperator{\Out}{Out}
\DeclareMathOperator{\Mod}{Mod}\DeclareMathOperator{\PMod}{PMod}
\DeclareMathOperator{\Inn}{Inn}

\DeclareMathOperator{\Push}{Push}

\newcommand{\vp}{\varphi}

\newcommand{\ol}{\overline}

\newcommand{\nid}{\noindent}

\newcommand{\innp}[1]{\left< #1 \right>}
\newcommand{\abs}[1]{\left\vert#1\right\vert}
\newcommand{\set}[1]{\left\{#1\right\}}

\newcommand{\co}{\colon}

\newcommand{\bop}{\bigoplus}

\newcommand{\lra}{\longrightarrow}

\newcommand{\B}[1]{\ensuremath{\mathbf{#1}}}

\newcommand{\Cal}[1]{\ensuremath{\mathcal{#1}}}

\usepackage{fancyhdr}

\fancypagestyle{plain}

\begin{document}

\title{The congruence subgroup problem for \\ pure braid groups: Thurston's proof}
\author{D. B. McReynolds\thanks{Partially supported by an NSF postdoctoral fellowship. In addition, part of this work was done while at the California Institute of Technology.}}
\maketitle


\begin{abstract}
\nid In this article we present an unpublished proof of W. Thurston that pure braid groups have the congruence subgroup property.
\end{abstract}

\section{Introduction}\label{Intro}

\nid Let $S_{g,n}$ denote a surface of genus $g$ with $n$ punctures. The \emph{pure mapping class group} $\PMod(S_{g,n})$ of $S_{g,n}$ is the subgroup of the group $\textrm{Diffeo}^+(S_{g,n})/\textrm{Diffeo}_0^+(S_{g,n})$ of orientation preserving diffeomorphisms that fix each puncture modulo isotopy; this definition differs from the notion of a pure mapping class when $g=0$. The Dehn--Nielsen Theorem (see \cite[Theorem 3.6]{FM} for instance) affords us with an injection of $\PMod(S_{g,n})$ into the outer automorphism group $\Out(\pi_1(S_{g,n}))$. Being a subgroup of $\Out(S_{g,n})$, the pure mapping class group $\PMod(S_{g,n})$ is endowed with a class of finite index subgroups called congruence subgroups. For each characteristic subgroup $K$ of $\pi_1(S_{g,n})$, we have an induced homomorphism $\PMod(S_{g,n}) \to \Out(\pi_1(S_{g,n}/K))$. When $K$ is finite index, the kernel of the induced homomorphism is a finite index subgroup of $\PMod(S_{g,n})$. These subgroups are called \emph{principal congruence subgroups} (see Section \ref{Sec2} for a more general discussion) and any finite index subgroup of $\PMod(S_{g,n})$ containing a principal congruence subgroup is called a \emph{congruence subgroup}. The purpose of this article is to address the following problem sometimes called the \emph{congruence subgroup problem} (see \cite{BL}, \cite{Ivanov}):

\nid \textbf{Congruence Subgroup Problem}. \emph{Is every finite index subgroup of $\PMod(S_{g,n})$ a congruence subgroup?}

\nid The congruence subgroup problem for pure mapping class groups $\PMod(S_{g,n})$ is a central problem for understanding $\Mod(S_{g,n})$ and $\PMod(S_{g,n})$. A positive answer allows one a means of understanding the finite index subgroup structure of $\Mod(S_{g,n})$ and thus profinite completion of $\Mod(S_{g,n})$. A few potential applications are a more precise understanding of the subgroup growth asymptotics for $\Mod(S_{g,n})$ and a better understanding of the absolute Galois group $\textrm{Gal}(\ol{\mathbf{Q}}/\mathbf{Q})$ via its action on the profinite completion of $\Mod(S_{g,n})$. The first case to be resolved was for $g=0, n>0$ by Diaz--Donagi--Harbater \cite{Num} in 1989 (though explicitly stated in the article only for $n=4$). Asada \cite[Theorem 3A, Theorem 5]{Asada} gave a proof for $g=0,1$ and $n>0$ in 2001 (for $g=1$, see also \cite{BER} and \cite{EM}). Boggi \cite[Theorem 6.1]{Boggi1} claimed a general solution to the congruence subgroup problem in 2006. However, a gap in \cite[Theorem 5.4]{Boggi1} was discovered by Abromovich, Kent, and Wieland\footnote{The gap was discovered by D. Abromovich, R. Kent, and B. Wieland while Abromovich prepared a review of this article for MathSciNet. They informed Boggi of the gap which he acknowledged in \cite[p. 3]{Boggi2}.} (see the forthcoming articles \cite{Kent,BK} for more on this). Boggi \cite[Theorem 3.5]{Boggi2} has since claimed to handle the cases of $g=0,1,2$ (with $n>0,n>0,n\geq 0$, resp.). All of these proofs are in the language of algebraic geometry, field extensions, and profinite groups. In contrast, in 2002 W. Thurston \cite{Thurston} outlined an explicit, elementary proof for $g=0$ that followed the general strategy given in \cite{Asada, Boggi1,Boggi2}. This article gives a detailed account based on \cite{Thurston}. For future reference, we state the result here.

\begin{thm}\label{Main}
$\PMod(S_{0,n})$ has the congruence subgroup property.
\end{thm}

\nid A few words are in order on how Thurston's proof compares to the proofs of Asada and Boggi. The proofs of Asada and Boggi are both short and elegant but use the language of profinite groups. Thurston's proof is longer but avoids the use of profinite groups and is essentially an explicit version of the proofs of Asada and Boggi. All three use the Birman exact sequence and use the fact that certain groups are centerless to control what one might call exceptional symmetries. All three use a homomorphism $\delta$ introduced below for this task. The merit of Thurston's proof is it's elementary nature; aside from Birman's work, the proof uses only elementary group theory. \smallskip\smallskip

\nid The second goal of this article is to introduce to a larger audience the simplicity of this result, be it Asada, Boggi, or Thurston's proof (see \cite{EM} for a better introduction to Diaz--Donagi--Harbater \cite{Num}). In addition, we hope to spark more interest in the general congruence subgroup problem for mapping class groups, a problem that is substantially more difficult than the simple case addressed here. Finally, we hope that those less familiar with the tools used in Asada, Boggi, and Diaz--Donagi--Harbater will see the potential for their methods, as in comparison to Thurston's proof, they provide a very simple and elegant framework for this problem.

\paragraph{Acknowledgements}

\nid I would first like to thank Nathan Dunfield for sharing with me Thurston's ideas. Most of my knowledge on this subject was gained from conversations with Dunfield and Chris Leininger, and I am deeply appreciative of the time both gave to me on this topic. I would also like to acknowledge the hard work of Dan Abromovich, Richard Kent IV, and Ben Wieland on reading \cite{Boggi1}. I would like to give Kent special thanks for several conversations on this article and on \cite{Boggi1, Boggi2}. I would like to thank Jordan Ellenberg for pointing out \cite{Num}, and Tom Church, Ellenberg, Benson Farb, Kent, Andy Putman, Justin Sinz, and the referees for several useful and indispensable comments on this article. Finally, I would like to thank Bill Thurston for allowing me to use the ideas presented in this article.

\section{Preliminaries}\label{Sec2}

\nid For a group $G$, the automorphism group of $G$ will be denoted by $\Aut(G)$. The normal subgroup of inner automorphisms will be denoted by $\Inn(G)$, and the group of outer automorphisms $\Aut(G)/\Inn(G)$ will be denoted by $\Out(G)$. For an element $g \in G$, the $G$--conjugacy class of $g$ will be denoted by $[g]$. The subgroup of $G$ generated by a set of elements $g_1,\dots,g_r$ will be denoted by $\innp{g_1,\dots,g_r}$. The center of $G$ will be denoted by $Z(G)$ and the centralizer of an element $g$ will be denoted by $C_G(g)$.

\paragraph{1. Congruence subgroups}

\nid Let $G$ be a finitely generated group and $\Lambda$ a subgroup of $\Aut(G)$ (resp. $\Out(G)$). We say that a normal subgroup $H$ of $G$ is \emph{$\Lambda$--invariant} if $\lambda(H)<H$ for all $\lambda$ in $\Lambda$. For such a subgroup, the canonical epimorphism
\[ \rho_H\co G \lra G/H \]
induces a homomorphism
\[ \rho_H^\star\co \Lambda \lra \Aut(G/H) \quad \text{(resp. }\rho_H^\ast\co \Lambda \lra \Out(G/H)\text{)} \]
defined via the formula
\[ \rho_H^\star(\psi)(gH) = \psi(g)H. \]
When $H$ is finite index, $\ker \rho_H^\star$ (resp. $\ker \rho_H^\ast$) is finite index in $\Lambda$ and is called a \emph{principal congruence subgroup}. Any subgroup of $\Lambda$ that contains a principal congruence subgroup is called a \emph{congruence subgroup}. We say that $\Lambda$ has the \emph{congruence subgroup property} if every finite index subgroup of $\Lambda$ is a congruence subgroup (see Bass--Lubotzky \cite{BL} for other examples of congruence subgroup problems).\smallskip\smallskip

\nid The following lemma will be useful throughout this article.

\begin{lemma}\label{L1}
The finite intersection of congruence subgroups is a congruence subgroup.
\end{lemma}

\begin{pf}
It suffices to prove the lemma for a pair of principal congruence subgroups $\ker \rho_{\Gamma_1}^\star$ and $\ker \rho_{\Gamma_2}^\star$. Set $Q_j = \Gamma/\Gamma_j$, $\Delta = \Gamma_1 \cap \Gamma_2$, and $Q = \Gamma/\Delta$. We have the commutative diagram
\begin{equation}\label{CongruenceIntersectionDiagram}
\xymatrix{ & \Gamma \ar[dd]^{\rho_\Delta} \ar[ld]_{\rho_{\Gamma_1}} \ar[rd]^{\rho_{\Gamma_2}} & \\ Q_1 &  & Q_2 \\ & Q, \ar[ru]_{\pi_2} \ar[lu]^{\pi_1} &}
\end{equation}
where the maps
\[ \pi_j\colon Q < Q_1 \times Q_2 \lra Q_j \]
are projection onto the $j$th factor. By definition
\[ \rho_\Delta^\star(\tau)(\gamma \Delta) = \rho_\Delta(\tau(\gamma))\Delta. \]
According to (\ref{CongruenceIntersectionDiagram}), we have
\[ \rho_\Delta(\gamma) = (\rho_{\Gamma_1}(\gamma),\rho_{\Gamma_2}(\gamma)). \]
If $\tau \in \ker \rho_{\Gamma_1}^\star \cap \rho_{\Gamma_2}^\star$, then
\[ (\rho_{\Gamma_1}(\tau(\gamma)),\rho_{\Gamma_2}(\tau(\gamma))) = (\rho_{\Gamma_1}(\gamma),\rho_{\Gamma_2}(\gamma)). \]
Therefore, $\tau \in \ker \rho_\Delta^\star$ and so
\[ \ker \rho_\Delta^\star < \ker \rho_{\Gamma_1}^\star \cap \ker \rho_{\Gamma_2}^\star. \]
The case when $\Lambda<\Out(\Gamma)$ is similar and yields the containment
\[ \ker \rho_\Delta^\ast < \ker \rho_{\Gamma_1}^\ast \cap \ker \rho_{\Delta_2}^\ast. \]
\end{pf}

\paragraph{2. Geometrically characteristic subgroups}

\nid For $\Lambda = \PMod(S_{g,n})$ and $G=\pi_1(S_{g,n})$, we call $\PMod(S_{g,n})$--invariant subgroups of $\pi_1(S_{g,n})$ \emph{geometrically characteristic subgroups}. We will denote the elements of $\pi_1(S_{g,n})$ generated by simple loops about the $n$ punctures by $\gamma_1,\dots,\gamma_n$. The subgroup of $\Aut(S_{g,n})$ that fixes each conjugacy class $[\gamma_j]$ will be denoted by $\Aut_c(\pi_1(S_{g,n}))$ and we set
\[ \Out_c(\pi_1(S_{g,n})) = \Aut_c(\pi_1(S_{g,n}))/\Inn(\pi_1(S_{g,n})). \]
The image of the pure mapping class group $\PMod(S_{g,n})$ afforded by the Dehn--Nielsen Theorem is a subgroup of $\Out_c(\pi_1(S_{g,n}))$. In the case when $g=0$, we list only the elements $\gamma_1,\dots,\gamma_{n-1}$ generated by simple loops about the punctures. For notational simplicity, we single out the element $\gamma_n$ (or $\gamma_{n-1}$ in the case $g=0$) and denote it simply by $\lambda$.

\paragraph{3. The Birman exact sequences}

\nid The normal closure of $\innp{\lambda}$ will be denoted by $N_\lambda$ and yields the short exact sequence
\[ \xymatrix{ 1 \ar[r] & N_\lambda \ar[r] & \pi_1(S_{g,n}) \ar[r]^{\rho_{N_\lambda}} & \pi_1(S_{g,n-1}) \ar[r] & 1}. \]
Since $N_\lambda$ is $\PMod(S_{g,n})$--invariant, $N_\lambda$ is geometrically characteristic and induces a short exact sequence
\[ \xymatrix{ 1 \ar[r] & K_\lambda \ar[r] & \PMod(S_{g,n}) \ar[r]^{\rho_{N_\lambda}^\ast} & \PMod(S_{g,n-1}) \ar[r] & 1}. \]
We also have the sequence
\begin{equation}\label{InnerSequence}
\xymatrix{1 \ar[r] & \pi_1(S_{g,n-1}) \ar[r]^\mu & \Aut_c(\pi_1(S_{g,n-1})) \ar[r]^{\theta} & \Out_c(\pi_1(S_{g,n-1})) \ar[r] & 1,}
\end{equation}
where $\mu(\eta)$ is the associated inner automorphism given by conjugation by $\eta$. These two sequences are related via a homomorphism
\[ \delta\colon \Out_c(\pi_1(S_{g,n})) \lra \Aut_c(\pi_1(S_{g,n-1})). \]
The map $\delta$ is given as follows. First, we select a normalized section of $\theta$
\[ s\colon \Out_c(\pi_1(S_{g,n})) \lra \Aut_c(\pi_1(S_{g,n})) \]
by sending an outer automorphism $\tau$ to an automorphism $s(\tau)$ such that $s(\tau)(\lambda) = \lambda$. The selection of $s(\tau)$ is unique up to right multiplication by the subgroup $\innp{\mu(\lambda)}$ of $\Inn(\pi_1(S_{g,n}))$. As $N_\lambda$ is $\Aut_c(\pi_1(S_{g,n}))$--invariant, we have an induced homomorphism
\[ \rho_{N_\lambda}^\star \colon \Aut_c(\pi_1(S_{g,n})) \lra \Aut_c(\pi_1(S_{g,n-1})), \]
and define $\delta$ by
\[ \delta(\tau) = \rho_{N_\lambda}^\star(s(\tau)). \]
Since the choice of $s$ is unique up to multiplication by the subgroup $\innp{\mu(\lambda)}$ and $\rho_{N_\lambda}^\star(\mu(\lambda)) = 1$, the map $\delta$ is a homomorphism. Under $\delta$, the subgroup $K_\lambda$ must map into $\Inn(\pi_1(S_{g,n-1}))$ since the projection to $\Out_c(\pi_1(S_{g,n-1}))$ is trivial. In fact, there exists an isomorphism
\[ \textrm{Push}\colon \pi_1(S_{g,n-1}) \lra K_\lambda, \]
and the result is the Birman exact sequence (see \cite[Theorem 4.5]{FM} for instance)
\begin{equation}\label{2}
\xymatrix{ 1 \ar[r] & \pi_1(S_{g,n-1}) \ar[r]^\Push & \PMod(S_{g,n}) \ar[r]^{\rho_{N_\lambda}^\ast} & \PMod(S_{g,n-1}) \ar[r] & 1.}
\end{equation}
The aforementioned relationship between the sequences (\ref{InnerSequence}) and (\ref{2}) given by $\delta$ is the content of our next lemma (see for instance \cite[p. 130]{Asada}).

\begin{lemma}\label{SequenceConnectionLemma}
$\mu = \delta \circ \Push$.
\end{lemma}

\nid Lemma \ref{SequenceConnectionLemma} is well known and there are several ways to prove it. One proof is to check by direct computation that $\delta \circ \Push = \mu$. This can be done explicitly by verifying this functional equation for a standard generating set for $\pi_1(S_{g,n-1})$.\smallskip\smallskip

\nid We finish this section with the following useful lemma.

\begin{lemma}\label{L2}
The $\rho_{N_\lambda}^\ast$--pullback of a congruence subgroup is a congruence subgroup.
\end{lemma}

\begin{pf}
Given a principal congruence subgroup $\ker \rho_\Delta^\ast$ of $\PMod(S_{g,n-1})$ with associated geometrically characteristic subgroup $\Delta$ of $\pi_1(S_{g,n-1})$, the subgroup $\rho_{N_\lambda}^{-1}(\Delta)$ is a geometrically characteristic subgroup of $\pi_1(S_{g,n})$. The associated principal congruence is the $\rho_{N_\lambda}^\ast$--pullback of $\ker \rho_\Delta^\ast$.
\end{pf}

\section{Proof of Theorem \ref{Main}}

\nid The first and main step in proving Theorem \ref{Main} is the following (see also \cite[Lemma 2.6]{LS} for another proof of this proposition).

\begin{prop}\label{MainProp}
$\Push(\pi_1(S_{0,n-1}))$ has the congruence subgroup property.
\end{prop}

\nid Using Proposition \ref{MainProp}, we will deduce the following inductive result, which is the second step in proving Theorem \ref{Main}.

\begin{prop}\label{InductionProp}
If $\PMod(S_{0,n-1})$ has the congruence subgroup property, then $\PMod(S_{0,n})$ has the congruence subgroup property.
\end{prop}

\nid We now give a quick proof of Theorem \ref{Main} assuming these results.

\begin{pf}[Proof of Theorem \ref{Main}]
The first non-trivial case occurs when $n=4$ where Proposition \ref{MainProp} and (\ref{2}) establish that $\PMod(S_{0,4})$ has the congruence subgroup property. Specifically, $\Push(\pi_1(S_{0,3})) = \PMod(S_{0,4})$. From this equality, one obtains Theorem \ref{Main} by employing Proposition \ref{InductionProp} inductively.
\end{pf}

\section{Proof of Proposition \ref{InductionProp}}

\nid As the proof of Proposition \ref{InductionProp} only requires the statement of Proposition \ref{MainProp}, we prove Proposition \ref{InductionProp} before commencing with the proof of Proposition \ref{MainProp}.

\begin{proof}[Proof of Proposition \ref{InductionProp}]
Given a finite index subgroup $\Lambda$ of $\PMod(S_{0,n})$, by passing to a normal finite index subgroup $\ker q < \Lambda$, it suffices to prove that $\ker q$ is a congruence subgroup. From $\ker q$, we obtain a surjective homomorphism
\[ q\co \PMod(S_{0,n}) \lra Q. \]
We decomposition $Q$ via the Birman exact sequence. Specifically, the Birman exact sequence (\ref{2}) produces a diagram
\begin{equation}\label{3}
\xymatrix{ 1 \ar[r] & \pi_1(S_{0,n-1}) \ar[r]^\Push \ar[d]_p & \PMod(S_{0,n}) \ar[r]^{\rho_{N_\lambda}^\ast} \ar[d]_q & \PMod(S_{0,n-1}) \ar[d]_r \ar[r] & 1 \\ 1 \ar[r] & P \ar[r] & Q \ar[r] & R \ar[r] & 1. }
\end{equation}
Note that since this diagram is induced from the Birman sequence, both $p,r$ are surjective homomorphisms though possibly trivial. According to Proposition \ref{MainProp}, there exists a homomorphism
\[ \rho_\Gamma\co \pi_1(S_{0,n}) \lra \pi_1(S_{0,n})/\Gamma \]
with finite index, geometrically characteristic kernel $\Gamma$ such that
\begin{equation}\label{CongruenceContainment}
\ker(\rho_\Gamma^\ast \circ \Push) < \Push(\ker p).
\end{equation}
As we are assuming $\PMod(S_{0,n-1})$ has the congruence subgroup property, by Lemma \ref{L1} and Lemma \ref{L2}, it suffices to find a finite index subgroup $\Lambda_0$ of $\PMod(S_{0,n-1})$ such that $\ker \rho_\Gamma^\ast \cap (\rho_{N_\lambda}^\ast)^{-1}(\Lambda_0)<\ker q$. The subgroup $\Lambda_0 = \rho_{N_\lambda}^\ast(\ker \rho_\Gamma^\ast \cap \ker q)$ is our candidate. We assert that
\[ \ker \rho_\Gamma^\ast \cap (\rho_{N_\lambda}^\ast)^{-1}(\rho_{N_\lambda}^\ast(\ker \rho_\Gamma^\ast \cap \ker q)) < \ker q. \]
To see this containment, first note that
\[ (\rho_{N_\lambda}^\ast)^{-1}(\rho_{N_\lambda}^\ast(\ker \rho_\Gamma^\ast \cap \ker q)) = \Push(\pi_1(S_{0,n-1}))\cdot(\ker \rho_\Gamma^\ast \cap \ker q). \]
Every element in the latter subgroup can be written in the form $sk$ where $s$ is an element of $\Push(\pi_1(S_{0,n-1}))$ and $k$ is an element of $\ker \rho_\Gamma^\ast \cap \ker q$. If $\gamma$ is an element of
\[ \ker \rho_\Gamma^\ast \cap (\Push(\pi_1(S_{0,n-1}))\cdot(\ker \rho_\Gamma^\ast \cap \ker q)), \]
then writing $\gamma = sk$, we see that since both $sk$ and $k$ are elements of $\ker \rho_\Gamma^\ast$, then so is $s$. In particular, it must be that $s$ is an element of $\ker \rho_\Gamma^\ast \cap \Push(\pi_1(S_{0,n-1}))$. By (\ref{CongruenceContainment}), we have
\[ \ker(\rho_\Gamma^\ast \circ \Push) = \ker \rho_\Gamma^\ast \cap \Push(\pi_1(S_{0,n-1})) < \Push(\ker p), \]
and so $s$ is an element of $\Push(\ker p)$. Therefore, we now know that
\[ \ker \rho_\Gamma^\ast \cap (\rho_{N_\lambda}^\ast)^{-1}(\rho_{N_\lambda}^\ast(\ker \rho_\Gamma^\ast \cap \ker q)) < \Push(\ker p)\cdot(\ker \rho_\Gamma^\ast \cap \ker q). \]
Visibly, any element of $\Push(\ker p) \cdot (\ker \rho_\Gamma^\ast \cap \ker q)$ is an element of $\ker q$, and so we have
\[ \ker \rho_\Gamma^\ast \cap (\rho_{N_\lambda}^\ast)^{-1}(\rho_{N_\lambda}^\ast(\ker \rho_\Gamma^\ast \cap \ker q)) < \ker q \]
as needed.
\end{proof}

\nid We note that the above proof makes no use of the assumption $g=0$, provided one knows Proposition \ref{MainProp} for $\Push(\pi_1(S_{g,n-1}))$.

\section{Proof of Proposition \ref{MainProp}}

\nid We now prove Proposition \ref{MainProp}. The proof is split into two steps. First, we reduce Proposition \ref{MainProp} to a purely group theoretic problem using Lemma \ref{SequenceConnectionLemma}. Using elementary methods, we then solve the associated group theoretic problem. Keep in mind that one of our main goals is keeping the proof of Theorem \ref{Main} as elementary as possible by which we mean to minimize the sophistication level of the mathematics involved and avoiding using results whose proofs require mathematics beyond an undergraduate algebra course. The trade off is that our arguments are longer. A good example of this trade off is our proof of Lemma \ref{CenterlessLemma} in comparison to \cite[Lemma 1]{Asada}, \cite[Lemma 2.6]{Boggi2}, or \cite[Proposition 2.7]{LS}.\smallskip\smallskip

\nid Given a finite index subgroup $\Gamma$ of $\pi_1(S_{0,n-1})$, we first pass to a finite index normal subgroup $\ker p$ of $\Gamma$ with associated homomorphism $p\colon \pi_1(S_{0,n-1}) \to P$. It suffices to show that $\ker p$ is a congruence subgroup and this will now be our goal.

\nid \textbf{Step 1.} We first describe congruence subgroups in $\pi_1(S_{0,n-1})$. Given a geometrically characteristic subgroup $\ker q$ of $\pi_1(S_{0,n})$ with associated homomorphism $q\colon \pi_1(S_{0,n}) \to Q$, we obtain a geometrically characteristic subgroup $\ker p_0$ of $\pi_1(S_{0,n-1})$ via the commutative diagram
\[ \xymatrix{ \pi_1(S_{0,n}) \ar[rr]^{\rho_{N_\lambda}} \ar[d]_q & & \pi_1(S_{0,n-1}) \ar[d]^{p_0} \\ Q \ar[rr]_{\rho_{q(N_\lambda)}} & & Q/q(N_\lambda) = P_0.} \]
In addition, we have the homomorphism
\[ \overline{\mu}\colon P_0 \lra \Inn(P_0) = P_0/Z(P_0), \]
where $Z(P_0)$ is the center of $P_0$. We would like, as before, to define a homomorphism
\[ \overline{\delta}\colon \Out_c(Q) \lra \Aut_c(P_0) \]
that relates $\rho_{\ker q}^\ast \circ \Push$ and $\overline{\mu} \circ p_0$. Proceeding as before, we define the map
\[ \overline{\delta}\colon \Out_c(Q) \lra \Aut_c(P_0). \]
Unfortunately, $\overline{\delta}$ need not be a homomorphism. To be precise, we set $\Aut_c(Q)$ to be the subgroup of $\Aut(Q)$ of automorphisms that preserve the conjugacy classes $[q(\gamma_1)],\dots[q(\gamma_{n-2})],[q(\lambda)]$ and $\Out_c(Q) = \Aut_c(Q)/\Inn(Q)$. Similarly, $\Aut_c(P_0)$ is the subgroup of $\Aut(P_0)$ that preserve the classes $[p_0(\gamma_1)],\dots,[p_0(\gamma_{n-2})]$. We take a normalized section
\[ \overline{s}\colon \Out_c(Q) \lra \Aut_c(Q) \]
by mandating that $s(\tau)(q(\lambda)) = q(\lambda)$ and then apply the homomorphism
\[ \rho_{q(N_\lambda)}^\star\colon \Aut_c(Q) \lra \Aut_c(P_0) \]
induced by the homomorphism $\rho_{q(N_\lambda)}$. The ambiguity in the selection of the section $\overline{s}$ is up to multiplication by the subgroup $\overline{\mu}(C_Q(q(\lambda)))$ of $\Inn(Q)$, the image of the centralizer of $q(\lambda)$ in $Q$ under $\overline{\mu}$. Provided $C_Q(q(\lambda))$ maps to the trivial subgroup under $\rho_{q(N_\lambda)}$, the resulting map
\[ \overline{\delta}\colon \Out_c(Q) \lra \Aut_c(P_0) \]
given by $\overline{\delta} = \rho_{q(N_\lambda)}^\star \circ \overline{s}$ is a homomorphism.

\begin{lemma}\label{RedEquLemma}
If $C_Q(q(\lambda)) < \ker \rho_{q(N_\lambda)}$, then $\overline{\delta} \circ \rho_{\ker q}^\ast \circ \Push = \overline{\mu} \circ p_0$.
\end{lemma}

\begin{pf}
The essence of this lemma is that in the diagram
\begin{equation}\label{BIGD}
\xymatrix{ & & & \Out_c(Q) \ar[d]^{\overline{s}} \ar[ddr]^{\overline{\delta}}& \\ & \Out_c(\pi_1(S_{0,n})) \ar[r]_{s} \ar[rd]_\delta \ar[rru]^{\rho_{\ker q}^\ast} & \Aut_c(\pi_1(S_{0,n})) \ar[r]^{\rho_{\ker q}^\star} \ar[d]^{\rho_{N_\lambda}^\star} & \Aut_c(Q) \ar[rd]_{\rho_{q(N_\lambda)}^\star} & \\ \pi_1(S_{0,n-1}) \ar[ru]^{\Push} \ar[rr]^{\mu} \ar[rd]_{p_0} & & \Aut_c(\pi_1(S_{0,n-1})) \ar[rr]^{\rho_{\ker p_0}^\star} & & \Aut_c(P_0) \\ & P_0 \ar[rrru]_{\overline{\mu}} & & & }
\end{equation}
we can push the bottom map $\overline{\mu} \circ p_0$ through to the top map $\overline{\delta} \circ \rho_{\ker q}^\ast \circ \Push$. The chief difficulty in proving this assertion is the non-commutativity of the top most right triangle ((\ref{CD1}) below) in (\ref{BIGD}). To begin, the diagrams
\begin{equation}\label{CD3}
\xymatrix{ \pi_1(S_{0,n-1}) \ar[d]_{p_0} \ar[rr]^\mu & & \Aut_c(\pi_1(S_{0,n-1})) \ar[d]^{\rho_{\ker p_0}} \\ P_0 \ar[rr]_{\overline{\mu}} & & \Aut_c(P_0)}
\end{equation}
and
\begin{equation}\label{CD2}
\xymatrix{\Aut_c(\pi_1(S_{0,n})) \ar[d]_{\rho_{\ker q}^\star} \ar[rr]^{\rho_{N_\lambda}^\star} & & \Aut_c(\pi_1(S_{0,n-1})) \ar[d]^{\rho_{\ker p_0}^\star} \\ \Aut_c(Q) \ar[rr]_{\rho_{q(N_\lambda)}^\star} & & \Aut_c(P_0)} \end{equation}
commute. The commutativity of (\ref{CD3}) and (\ref{CD2}) in tandem with Lemma \ref{SequenceConnectionLemma} yield the following string of functional equalities:
\begin{align*}
\overline{\mu} \circ p_0 &= \rho_{\ker p_0}^\star \circ \mu & (\text{by }(\ref{CD3})) \\
&= \rho_{\ker p_0}^\star \circ \delta \circ \Push & (\text{by Lemma }\ref{SequenceConnectionLemma}) \\
&= \rho_{\ker p_0}^\star \circ \rho_{N_\lambda}^\star \circ s \circ \Push & (\text{by definition of }\delta) \\
&= \rho_{q(N_\lambda)}^\star \circ \rho_{\ker q}^\star \circ s \circ \Push & (\text{by }(\ref{CD2})).
\end{align*}
We claim that
\begin{equation}\label{FunctionEquation}
\rho_{q(N_\lambda)}^\star \circ \overline{s} \circ \rho_{\ker q}^\ast = \rho_{q(N_\lambda)}^\star \circ \rho_{\ker q}^\star \circ s
\end{equation}
holds. However, since the diagram
\begin{equation}\label{CD1}
\xymatrix{ \Aut_c(\pi_1(S_{0,n})) \ar[rr]^{\rho_{\ker q}^\star} & &\Aut_c(Q) \\ \Out_c(\pi_1(S_{0,n})) \ar[u]^s \ar[rr]_{\rho_{\ker q}^\ast} & & \Out_c(Q) \ar[u]_{\overline{s}} }
\end{equation}
need not commute, to show (\ref{FunctionEquation}), we must understand the failure of (\ref{CD1}) to commute. Note that the validity of (\ref{FunctionEquation}) amounts to showing the failure of the commutativity of (\ref{CD1}), namely $(\rho_{\ker q}^\star(s(\tau)))^{-1}\overline{s}(\rho_{\ker q}^\ast(\tau))$, resides in the kernel of $\rho_{q(N_\lambda)}^\star$. To that end, set
\[ \theta\colon \Aut_c(\pi_1(S_{0,n})) \lra \Out_c(\pi_1(S_{0,n})) \]
and
\[ \overline{\theta}\colon \Aut_c(Q) \lra \Out_c(Q) \]
to be the homomorphisms induced by reduction modulo the subgroups $\Inn(\pi_1(S_{0,n}))$ and $\Inn(Q)$, respectively. As $s$ and $\overline{s}$ are normalized sections of $\theta$ and $\overline{\theta}$, we have
\begin{equation}\label{SectionEquation}
\theta \circ s = \textrm{Id}, \quad \overline{\theta} \circ \overline{s} = \textrm{Id}.
\end{equation}
The commutativity of the diagram
\[ \xymatrix{ \Aut_c(\pi_1(S_{0,n})) \ar[rr]^{\rho_{\ker q}^\star} \ar[d]_{\theta} & &\Aut_c(Q) \ar[d]^{\overline{\theta}} \\ \Out_c(\pi_1(S_{0,n})) \ar[rr]_{\rho_{\ker q}^\ast} & & \Out_c(Q)}  \]
with (\ref{SectionEquation}) yields
\begin{equation}\label{ThetaProjection}
\overline{\theta} \circ \rho_{\ker q}^\star \circ s = \rho_{\ker q}^\ast, \quad \overline{\theta} \circ \overline{s} \circ \rho_{\ker q}^\ast = \rho_{\ker q}^\ast.
\end{equation}
Since $s(\tau)(\lambda) = \lambda$ and $\overline{s}(\overline{\tau})(q(\lambda)) = q(\lambda)$, we also have
\[ \rho_{\ker q}^\star(s(\tau))(q(\lambda)) = \overline{s}(\rho_{\ker q}^\ast(\tau))(q(\lambda)) = q(\lambda). \]
This equality in combination with (\ref{ThetaProjection}) imply that $\rho_{\ker q}^\star(s(\tau))$ and $\overline{s}(\rho_{\ker q}^\ast(\tau))$ differ by multiplication by an element of $\overline{\mu}(C_Q(q(\lambda)))$. Equivalently, the element
\[ (\rho_{\ker q}^\star(s(\tau)))^{-1}\overline{s}(\rho_{\ker q}^\ast(\tau)), \]
which measures the failure of the commutativity of (\ref{CD1}), resides in the subgroup $\overline{\mu}(C_Q(q(\lambda)))$. However, by assumption, $C_Q(q(\lambda)) < \ker \rho_{q(N_\lambda)}$ and so we have the equality claimed in (\ref{FunctionEquation}). Continuing our string of functional equalities started prior to (\ref{FunctionEquation}), the following string of functional equalities completes the proof:
\begin{align*}
\overline{\mu} \circ p_0 &= \rho_{q(N_\lambda)}^\star \circ \rho_{\ker q}^\star \circ s \circ \Push & (\text{by the computation above}) \\
&= \rho_{q(N_\lambda)}^\star \circ \overline{s} \circ \rho_{\ker q}^\ast \circ \Push & (\text{by }(\ref{FunctionEquation})) \\
&= \overline{\delta} \circ \rho_{\ker q}^\ast \circ \Push. & (\text{by definition of }\overline{\delta}).
\end{align*}
\end{pf}

\nid We say that a homomorphism $p_0\colon \pi_1(S_{0,n-1}) \to P_0$ is \emph{induced by $\pi_1(S_{0,n})$} if $p_0$ arises as above from a geometrically characteristic subgroup $\ker q$ of $\pi_1(S_{0,n})$ and $\overline{\delta}$ is a homomorphism. Under these assumptions, by Lemma \ref{RedEquLemma},
\[ \overline{\delta} \circ \rho_{\ker q}^\ast \circ \Push = \overline{\mu} \circ p_0. \]
Consequently,
\[ \ker(\rho_{\ker q}^\ast \circ \Push) < \ker(\overline{\mu} \circ p_0). \]
In particular, $\ker(\overline{\mu} \circ p_0)$ is a congruence subgroup and so the following lemma suffices for proving Proposition \ref{MainProp}.

\begin{lemma}\label{MainGroupLemma}
Let $\ker p$ be a finite index normal subgroup of $\pi_1(S_{0,n-1})$. Then there exists a homomorphism $p_0$ induced by $\pi_1(S_{0,n})$ such that $\ker(\overline{\mu} \circ p_0) < \ker p$.
\end{lemma}

\nid \textbf{Step 2.} The proof of Lemma \ref{MainGroupLemma} will also be split into two parts. This division is natural in the sense that we need to produce a homomorphism induced by $\pi_1(S_{0,n})$ and also control the center of the target of the induced homomorphism. We do the latter first via our next lemma as this lemma is only needed at the very end of the proof of Lemma \ref{MainGroupLemma}. In addition, some of the ideas used in the proof will be employed in the proof of Lemma \ref{MainGroupLemma} (see \cite[Proposition 2.7]{LS} for a more general result).

\begin{lemma}\label{CenterlessLemma}
Let $\ker p$ be a finite index normal subgroup of $\pi_1(S_{0,n-1})$ with $n>3$. Then there exists a finite index normal subgroup $\ker p_\ell$ of $\pi_1(S_{0,n-1})$ such that the resulting quotient $P_\ell$ is centerless and $\ker p_\ell < \ker p$.
\end{lemma}

\begin{pf}
We first pass to a normal subgroup $\ker q$ of $\ker p$ so that $\pi_1(S_{0,n-1})/\ker q$ is not cyclic. Note that if $P$ is cyclic, then $\ker p$ contains the kernel of the homology map $\pi_1(S_{0,n-1}) \to H_1(S_{0,n-1},\B{Z}/m\B{Z})$ for some $m$. We simply take this kernel for $\ker q$. Note that since $n>3$, the group $H_1(S_{0,n-1},\B{Z}/m\B{Z})$ is not cyclic. Let $Q$ denote the finite group $\pi_1(S_{0,n-1})/\ker q$. For a fixed prime $\ell$, let $V_\ell$ denote the $\B{F}_\ell$--group algebra of $Q$ where $\B{F}_\ell$ is the finite field of prime order $\ell$. Recall
\[ V_\ell = \set{\sum_{q' \in Q} \alpha_{q'} q', \quad \alpha_{q'} \in \B{F}_\ell} \]
is an $\B{F}_\ell$--vector space with basis $Q$ and algebra structure given by polynomial multiplication. The group $Q$ acts by left multiplication on $V_\ell$ and this action yields the split extension $V_\ell \rtimes Q$. Let $q(\gamma_j) = q_j$ and set $R_\ell$ to be the subgroup $V_\ell \rtimes Q$ generated by
\[ \set{(1,q_1),(0,q_2),\dots,(0,q_t)} = \set{r_1,r_2,\dots,r_{n-2}}. \]
We have a surjective homomorphism $r\colon \pi_1(S_{0,n-1}) \to R_\ell$ given by $r(\gamma_j) = r_j$. If $q_1$ has order $k_1$, note that
\[ r_1^{k_1} = (1,q_1)^{k_1} = (1+q_1+\dots+q_1^{k_1-1},1). \]
Now assume that $r' \in Z(R_\ell)$ is central and of the form $(v,q')$. It follows that $q' \in Z(Q)$ and
\[ v+q'(1+q_1+\dots+q_1^{k_1-1}) = v + (1+q_1+\dots+q_1^{k_1-1}). \]
Canceling $v$ from both side, we see that
\[ q'+q'q_1+\dots+q'q_1^{k_1-1} = 1 + q_1 + \dots + q_1^{k_1-1}. \]
In particular, there must be some power $k$ such that $q'q_1^k = 1$ and so $q' \in \innp{q_1}$. Next, set $W_\ell$ to be the $\B{F}_\ell$--group algebra of $R_\ell$ and let $S_\ell$ be the subgroup of $W_\ell \rtimes R_\ell$ generated by the set
\begin{align*}
\set{(1,(1,q_1)),(0,(0,q_2)),\dots,(0,(0,q_{n-2}))} &= \set{(1,r_1),(0,r_2),\dots,(0,r_{n-2})} \\
&= \set{s_1,s_2,\dots,s_{n-2}}.
\end{align*}
We again have a surjective homomorphism $s\colon \pi_1(S_{0,n-1}) \to S_\ell$ given by $s(\gamma_j) = s_j$. As before, if $s' \in Z(S_\ell)$ is central and of the form $(w,r')$, then $r' \in Z(R_\ell)$ and $r' \in \innp{r_1}$. In particular, for some $k \leq \abs{r_1}$, we have
\[ r' = (1+q_1+\dots+q_1^{k-1},q_1^{k}). \]
Since $r' \in Z(R_\ell)$, we have
\[ r_jr' = (0,q_j)r' = r'(0,q_j) = r'r_j \]
for $j>1$. This equality yields the equation
\[ q_j(1+q_1+\dots+q_1^{k-1}) = 1+q_1+\dots+q_1^{k-1}. \]
As before, this equality implies $q_j \in \innp{q_1}$ for all $j>1$ provided $k < \abs{r_1}$. However, if this holds, $Q$ must be cyclic. As $Q$ is non-cyclic, $r'$ must be trivial and $s'$ has the form $(w,0)$. For $(w,0)$ to be central in $S_\ell$, we must have
\[ (w,0)(0,r_j) = (0,r_j)(w,0) \]
for all $j \ne 1$ and
\[ (w,0)(1,r_1) = (1,r_1)(w,0). \]
These equalities imply that $r_jw = w$ for $j=1,\dots,n-2$. Since $\set{r_j}$ generate $R_\ell$, the element $w$ must be fixed by every element $r \in R_\ell$. However, the only vectors in $W_\ell$ that are fixed by every element of $R_\ell$ are of the form (see for instance \cite[p. 37]{FH})
\[ w_\alpha = \alpha \sum_{r \in R_\ell} r, \quad \alpha \in \B{F}_\ell. \]
Let $C$ be the normal cyclic subgroup $S_\ell \cap \innp{(w_1,0)}$ of $S_\ell$, $P_\ell = S_\ell/C$, and $p_{j,\ell}$ be the image of $s_j$ under this projection. By construction, $P_\ell$ is centerless and for the homomorphism
\[ p_\ell\colon \pi_1(S_{0,n-1}) \lra P_\ell \]
given by $p_\ell(\gamma_j)=p_{j,\ell}$, we have $\ker p_\ell < \ker p$. To see the latter, we simply note that we have the commutative diagram
\[ \xymatrix{ \pi_1(S_{0,n-1}) \ar[dd]_s \ar[rdd]^{p_\ell} \ar[rrrdd]_r \ar[rrrrrdd]_q \ar[rrrrrrrdd]^p & & & & & & &\\ & & & & & & &\\ S_\ell \ar[r] & P_\ell \ar[rr] & & R_\ell \ar[rr] & & Q \ar[rr] & & P,} \]
where the bottom maps are given by
\[ s_j \longmapsto p_{j,\ell} \longmapsto r_j \longmapsto q_j \longmapsto p_j = p(\gamma_j). \]
\end{pf}

\nid We are now ready to prove Lemma \ref{MainGroupLemma}.

\begin{pf}[Proof of Lemma \ref{MainGroupLemma}]
Given a finite index normal subgroup $\ker p$ of $\pi_1(S_{0,n-1})$, we must show that there is a homomorphism $p_0$ induced by $\pi_1(S_{0,n})$ such that $\ker(\overline{\mu} \circ p_0) < \ker p$. By Lemma \ref{CenterlessLemma}, we may assume that $P = \pi_1(S_{0,n-1})/\ker p$ is centerless. The homomorphism $p$ provides us with a homomorphism
\[ p \circ \rho_{N_\lambda}\colon \pi_1(S_{0,n}) \lra P. \]
Let $U_\ell$ be the $\B{F}_\ell$--group algebra of $P$ and define
\[ \vp\colon \pi_1(S_{0,n}) \lra U_\ell \rtimes P \]
by
\[ \vp(\lambda) = (1,p \circ \rho_{N_\lambda}(\lambda)), \quad \vp(\gamma_j) = (0,p\circ \rho_{N_\lambda}(\gamma_j)). \]
Note that the normal closure of $\vp(\lambda)$ contains the centralizer of $\vp(\lambda)$. Indeed, the normal closure of $\vp(\lambda)$ is simply $U_\ell \cap \vp(\pi_1(S_{0,n}))$. If $(v,p') \in \vp(\pi_1(S_{0,n}))$ commutes with $\vp(\lambda)$, then
\[ (v,p')(1,1) = (v+p',p') = (v+1,p') = (1,1)(v,p'). \]
Thus, $p'=1$ and $(v,p') \in U_\ell$. By construction, the diagram
\begin{equation}\label{CD4}
\xymatrix{ \pi_1(S_{0,n}) \ar[d]_{\vp} \ar[r]^{\rho_{N_\lambda}} & \pi_1(S_{0,n-1}) \ar[d]^p \\ U_\ell \rtimes P \ar[r]_{\rho_{U_\ell}} & P }
\end{equation}
commutes. This representation is unlikely to have a geometrically characteristic kernel. We rectify that as follows. Let $\Cal{O}_\vp$ denote the orbit of $\vp$ under the action of $\Aut_c(\pi_1(S_{0,n}))$ on $\textrm{Hom}(\pi_1(S_{0,n}),U_\ell \rtimes P)$ given by pre-composition. We define a new homomorphism
\[ q\co \pi_1(S_{0,n}) \lra Q < \bop_{\vp' \in \Cal{O}_\vp} U_\ell \rtimes P, \]
by
\[ q = \bop_{\vp' \in \Cal{O}_\vp} \vp'. \]
By construction, the kernel of this homomorphism is geometrically characteristic. In addition, each representation $\vp'$ has the property that the normal closure of $\vp'(\lambda)$ contains the centralizer of $\vp'(\lambda)$. Note that this follows from the fact that this containment holds for $\vp$ and the homomorphism $\vp'$ is equal to $\vp \circ \tau$ for some $\tau \in \Aut_c(\pi_1(S_{0,n}))$. As $\tau$ preserves the conjugacy class $[\lambda]$, $\vp'(\lambda)$ is conjugate to $\vp(\lambda)$ in $U_\ell \rtimes P$. We assert that we have the inclusion $\ker(\overline{\mu} \circ p_0) < \ker p$, where $p_0$ is induced by the diagram
\[ \xymatrix{ \pi_1(S_{0,n}) \ar[d]_q \ar[r]^{\rho_{N_\lambda}} & \pi_1(S_{0,n-1}) \ar[d]^{p_0} \\ Q \ar[r]_{\rho_{q(N_\lambda)}} & P_0. } \]
To see this containment, we first observe that
\[ Q/q(N_\lambda) < \bop_{\vp' \in \Cal{O}_\vp} \vp'(\pi_1(S_{0,n}))/\vp'(N_\lambda) = \bop_{\vp' \in \Cal{O}_\vp} P. \]
Composing with the projection map $\pi_\vp$ onto the factor $\vp(\pi_1(S_{0,n}))$ associated with $\vp$, we get the commutative diagram
\begin{equation}\label{CD5}
\xymatrix{ \pi_1(S_{0,n}) \ar[d]_{\rho_{N_\lambda}} \ar[r]^q & Q \ar[d]^{\rho_{q(N_\lambda)}} \ar[r]^{\pi_\vp} & U_\ell \rtimes P \ar[d]^{\rho_{U_\ell}} \\ \pi_1(S_{0,n-1}) \ar[r]_{p_0} & P_0 \ar[r]_{\pi_\vp} & P.}
\end{equation}
Now, if $\gamma \in \ker(\overline{\mu} \circ p_0)$, then $p_0(\gamma)$ is central in $P_0$. As central elements map to central elements under homomorphisms, $\pi_\vp(p_0(\gamma))$ must be central in $P$. However, by assumption $P$ is centerless and so $f(p_0(\gamma))=1$. Since $\pi_\vp \circ p_0 = p$ by (\ref{CD4}) and (\ref{CD5}), we see that $p(\gamma) = 1$. Therefore, $\ker(\overline{\mu} \circ p_0) < \ker p$.
\end{pf}

\nid The construction above only uses that $\pi_1(S_{0,n-1})$ is a free group in the proof of Lemma \ref{CenterlessLemma}. With more care, the same method used in the proof of Lemma \ref{CenterlessLemma} can be used to prove the following (for $g=1$, we must assume $n>1$)---this again follows from \cite[Proposition 2.7]{LS}.

\begin{lemma}\label{CenterlessLemmaSurfaceCase}
Let $\ker p$ be a finite index normal subgroup of $\pi_1(S_{g,n-1})$. Then there exists a finite index normal subgroup $\ker p_\ell$ of $\pi_1(S_{g,n-1})$ such that the resulting quotient $P_\ell$ is centerless and $\ker p_\ell < \ker p$.
\end{lemma}

\nid In total, this yields an elementary proof of the following---this also follows from \cite[Lemma 2.6]{LS}.

\begin{prop}\label{MainPropClosedCase}
$\Push(\pi_1(S_{g,n-1}))$ has the congruence subgroup property (when $g=1$, $n>1$).
\end{prop}

\nid Finally, since the proof of Proposition \ref{InductionProp} does not require $g=0$, we have an elementary proof of the following, which was also proved in \cite[Theorem 2]{Asada} and \cite[Proposition 2.3]{Boggi2}.

\begin{prop}\label{InductionPropHigherGenus}
If $\PMod(S_{g,n-1})$ has the congruence subgroup property, then $\PMod(S_{g,n})$ has the congruence subgroup property.
\end{prop}

\section{Comparison of the proofs}

\nid We conclude this article with a more detailed comparison of the proofs of Theorem \ref{Main}. Instead of using the homomorphisms $\overline{\delta}$ employed above, Asada extends the homomorphism $\delta$ to
\[ \widehat{\delta}\colon \Out_c(\widehat{\pi_1(S_{g,n})}) \lra \Aut_c(\widehat{\pi_1(S_{g,n-1})}). \]
The group $\Aut_c(\widehat{\pi_1(S_{g,n})})$, for any $n$, is the group of continuous automorphisms of the profinite completion $\widehat{\pi_1(S_{g,n})}$ that preserve the conjugacy classes $[\gamma_j]$ and $[\lambda]$. We set $\Out_c(\widehat{\pi_1(S_{g,n})}) = \Aut_c(\widehat{\pi_1(S_{g,n})})/\Inn(\widehat{\pi_1(S_{g,n})})$. This extension is defined as before, though some care is needed in showing $\widehat{\delta}$ is a homomorphism. The result of the construction of $\widehat{\delta}$ yields a relationship similar to Lemma \ref{SequenceConnectionLemma} and is equivalent to our Step 1. The final ingredient needed is the fact that $Z(\widehat{\pi_1(S_{g,n})})$ is trivial, which is equivalent to Lemma \ref{CenterlessLemma}. Indeed, we have a sequence
\[ \xymatrix{ \widehat{\pi_1(S_{g,n-1})} \ar[r]^{\widehat{\Push}} & \Out_c(\widehat{\pi_1(S_{g,n})}) \ar[r]^{\widehat{\rho_{N_\lambda}^\ast}} & \Out_c(\widehat{\pi_1(S_{g,n-1})}) \ar[r] & 1.} \]
Proposition \ref{MainProp} is equivalent to the injectivity of $\widehat{\Push}$. The homomorphism $\widehat{\delta}$ relates this profinite version of (\ref{2}) to the profinite version of (\ref{InnerSequence})
\[ \xymatrix{ \widehat{\pi_1(S_{g,n-1})} \ar[r]^{\widehat{\mu}} & \Aut_c(\widehat{\pi_1(S_{g,n-1})}) \ar[r] & \Out_c(\widehat{\pi_1(S_{g,n-1})}) \ar[r] & 1.} \]
Specifically, the relationship is
\begin{equation}\label{ProfiniteConnection}
\widehat{\mu} = \widehat{\delta} \circ \widehat{\Push}.
\end{equation}
Thus, the injectivity of $\widehat{\Push}$ follows from the triviality of $Z(\widehat{\pi_1(S_{g,n-1})})$. Note that it is not obvious that (\ref{ProfiniteConnection}) holds and this was established in \cite{NT}. The content of Step 1 and parts of Step 2 reprove (\ref{ProfiniteConnection}). Boggi's proof \cite[p. 4--5]{Boggi2} is essentially the same Asada's proof though with different language and different notation that might initially veil the similarities. His analysis of centralizers in $\widehat{\pi_1(S_{g,n})}$ is different as he makes use of cohomological dimension and Shapiro's Lemma. Like the other two proofs, he also makes use of the homomorphism $\widehat{\delta}$. To summarize, in all of the proofs mentioned above, the main thrust is the reduction of Proposition \ref{MainProp} to a group theoretic statement like Lemma \ref{MainGroupLemma} followed by an argument that controls centers like Lemma \ref{CenterlessLemma}.\smallskip\smallskip

\nid The proof given by Diaz--Donagi--Harbater \cite{Num} also requires control of symmetries and a generalization of a group theoretic analog of their proof is given in \cite{EM}. However, their proof is sufficiently different from the rest as it is more geometric in nature.\smallskip\smallskip

\nid Boggi's general framework for the congruence subgroup problem introduced in \cite{Boggi1} and \cite{Boggi2} is a step in resolving the congruence subgroup problem in general. Despite the gap in \cite{Boggi1}, his work has introduced new tools and also he proves results that may be of independent interest to algebraic geometers, geometric group theorists, and geometers. Those with interests in these fields should study his work at far greater depth than what has been presented in this article.


\noindent Department of Mathematics \\
University of Chicago \\
Chicago, IL 60637 \\
email: {\tt dmcreyn@math.uchicago.edu}



\begin{thebibliography}{9999}

\bibitem{Asada}
M.~Asada, \emph{The faithfulness of the monodromy representations associated with certain families of algebraic curves}, J. Pure Appl. Algebra \textbf{159} (2001), 123--147.

\bibitem{BL}
H.~Bass and A.~Lubotzky, \emph{Automorphisms of groups and of schemes of finite type}, Israel J. Math. \textbf{44} (1983), 1--22.

\bibitem{Boggi1}
M.~Boggi, \emph{Profinite Teichmüller theory}, Math. Nachr. \textbf{279} (2006), 953--987.

\bibitem{Boggi2}
M.~Boggi, \emph{The congruence subgroup property for the hyperelliptic modular group}, preprint (2008).

\bibitem{BER}
K.-U.~Bux, M.~Ershov, and A.~Rapinchuk,
\emph{The congruence subgroup property for $\Aut(F_2)$: a group-theoretic proof of Asada's theorem}, preprint.

\bibitem{Num}
S.~Diaz, R.~Donagi, and D.~Harbater, \emph{Every curve is a Hurwitz space}, Duke Math. J. \textbf{59} (1989), 737--746.

\bibitem{EM}
J.~S.~Ellenberg and D.~B.~McReynolds, \emph{Every curve is a Teichm\"{u}ller curve}, preprint.

\bibitem{FM}
B.~Farb and D.~Margalit, \emph{A primer on mapping class groups}.

\bibitem{FH}
W.~Fulton and J.~Harris, \emph{Representation Theory}, Springer-Verlag 1991.

\bibitem{Ivanov}
N.~V.~Ivanov, \emph{Fifteen problems about the mapping class groups}. Problems on mapping class groups and related topics, Proc. Sympos. Pure Math. \textbf{74} (2006), 71--80.

\bibitem{Kent}
R.~P.~Kent, \emph{Topology of profinite curve complexes of doubly–punctured surfaces}, in preparation.

\bibitem{BK}
R.~P.~Kent and B. Wieland, \emph{Approaching the congruence subgroup problem for mapping class groups}, in preparation.

\bibitem{LS}
A.~Lubotzky and Y.~Shalom, \emph{Finite representations in the unitary dual and ramanujan groups}, Contemporary Math. \textbf{347} (2004), 173--189.

\bibitem{NT}
H.~Nakamura and H.~Tsunogai, \emph{Some finiteness theorems on Galois centralizers in pro-$l$ mapping class groups},
J. Reine Angew. Math. \textbf{441} (1993), 115--144.

\bibitem{Thurston}
W.~P.~Thurston, \emph{Commmunication to N. Dunfield}, (2002).

\end{thebibliography}
\end{document}